\documentclass[11pt,reqno]{article}

\usepackage[utf8]{inputenc}
\usepackage{authblk}
\usepackage{amsmath}
\usepackage{amsthm}
\usepackage{amssymb}
\usepackage{latexsym}
\usepackage{amstext}
\usepackage{tikz}
\usepackage{multicol}

\usepackage{setspace}
\usepackage{parskip}

\usepackage{graphicx}

\usepackage{setspace}
\usepackage{centernot}
\usepackage{amsthm}
\usepackage{fullpage}
\usepackage{tikz}
\usepackage{amsmath,amssymb,amsfonts}

\DeclareMathOperator{\Av}{Av}
\DeclareMathOperator{\inv}{inv}
\DeclareMathOperator{\Fib}{Fib}

\usepackage{hyperref}

\theoremstyle{plain}
\newtheorem{theorem}{Theorem}[section]

\newtheorem{proposition}[theorem]{Proposition}

\theoremstyle{definition}
\newtheorem{definition}[theorem]{Definition}

\theoremstyle{remark}

\title{Statistics on Almost-Fibonacci Pattern-Avoiding Permutations}

\author[1]{Brody Lynch}
\author[1]{Yihan Qin}

\affil[1]{Carleton College}

\begin{document}
\maketitle

\begin{abstract}
We prove that $|\Av_n(231,312,1432)|$, $|\Av_n(312,321,1342)|$ $|\Av_n(231,312,4321,21543)|$, and $ |\Av_n(321,231,4123,21534)|$,  are all equal to $F_{n+1} - 1$ where $F_n$ is the $n$-th Fibonacci number using the convention $F_0 = F_1 = 1$ and $\Av_n(S)$ is the set of all permutations of length $n$ that avoid all of the patterns in the set $S$. To do this, we characterize the structures of the permutations in these sets in terms of Fibonacci permutations. Then, we further quantify the structures using statistics such as inversion number and a statistic that measures the length of Fibonacci subsequences. Finally, we encode these statistics in generating functions written in terms of the generating function for Fibonacci permutations. We use these generating functions to find analogs about recurrence relation and addition formulae of Fibonacci identities.
\end{abstract}

\section{Introduction}
We say two sequences $a_1a_2\ldots a_k$ and $b_1b_2\ldots b_k$ of positive integers are \emph{order isomorphic} whenever $a_i < a_j$ if and only if $b_i < b_j$ for all $1\leq i,j \leq k$. A sequence $\pi$ \emph{contains} a sequence (or pattern) $\sigma$ whenever $\pi$ has a subsequence that is order isomorphic to $\sigma$. A permutation \emph{avoids} a pattern whenever it does not contain that pattern. For any set $S$ of permutations and any non-negative integer $n$, we write $\Av(S)$ to denote the set of all permutations which avoid all of the permutations in $S$ and we write $\Av_n(S)$ to denote the set of permutations of length $n$ in $\Av(S)$. In this context, we call the elements of $S$ \emph{forbidden patterns}.

Simion and Schmidt \cite{Simion} showed that $|\Av_n(231,312,321)| = F_n$. We will call permutations of this form \emph{Fibonacci permutations}. This Fibonacci structure will show up in the structures of permutations in all of our sets of pattern-avoiding permutations. It is implicit in Simion and Schmidt's proof that every permutation in $\Av(231, 312, 321)$ is made up of consecutive decreasing subsequences of length less than or equal to two where every entry in each subsequence is greater than all entries in prior subsequences. We will soon mention notation to make this description less clunky.

Permutations of length $n$ are bijections from $\{1,2,\ldots,n\}$ to $\{1,2,\ldots,n\}$, so we can plot them on coordinate axes with the horizontal axis representing the position of the entry in a permutation and the vertical axis representing its value. Since permutations are functions, we can use their plots to describe our sets of pattern-avoiding permutations. Given any set of points in the plane, the associated \emph{geometric grid class} is the set of permutations whose graphs can be drawn on the set of points. This definition comes from \cite{grid}, though in our case, we do not need to add the stipulation that no two points on a horizontal or vertical line can be chosen because this is impossible given the sets of points we picked. In this paper, we use grid classes mostly for descriptive and visualization purposes. For example, the set of Fibonacci permutations is the grid class associated with the infinite version of Figure \ref{fig:Fibonacci}, in which there are infinitely many pairs of dots.

\begin{figure}
    \centering
    \begin{tikzpicture}
\draw (-.5,-.5) -- (-.5,5);
\draw (-.5,-.5) -- (5,-.5);
\draw (-.5,5) -- (5,5);
\draw (5,-.5) -- (5,5);
\filldraw[black] (0,0.5) circle (2pt);
\filldraw[black] (0.5,0) circle (2pt);
\filldraw[black] (1,1.5) circle (2pt);
\filldraw[black] (1.5,1) circle (2pt);
\filldraw[black] (2,2.5) circle (2pt);
\filldraw[black] (2.5,2) circle (2pt);
\filldraw[black] (3,3.5) circle (2pt);
\filldraw[black] (3.5,3) circle (2pt);
\filldraw[black] (4,4.5) circle (2pt);
\filldraw[black] (4.5,4) circle (2pt);
\end{tikzpicture}
    \caption{Fibonacci Grid Class}
    \label{fig:Fibonacci}
\end{figure}
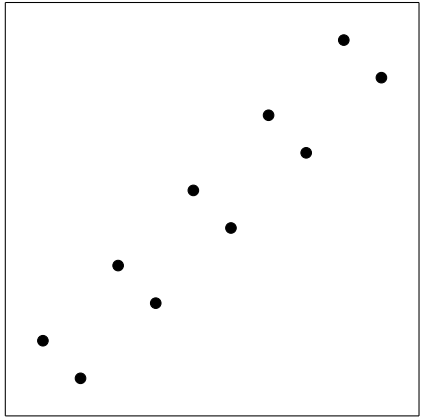

Our paper focuses on four sets of patterns, which we will label as
$$A_1 = \{231,312,4321,21543\}, A_2 = \{231,321,4123,21534\},$$ $$B_1 = \{231,312,1432\}, \text{ and } B_2 = \{312,321,1342\}.$$
It will soon become clear why they are grouped in this way.

In this paper, we will first use Fibonacci identities and tiling bijections to prove that for $n \geq 1$,
$$|\Av_n(A_1)| = |\Av_n(A_2)| = |\Av_n(B_1)| = |\Av_n(B_2)| = F_{n+1}-1.$$
This equality is somewhat remarkable, as it is an example of \textit{unbalanced Wilf-equivalence}, which is to say that $A_1$ and $A_2$ have different numbers and lengths of patterns than $B_1$ and $B_2$, but the number of permutations of a given length that avoid the patterns in each set is equal. Until recently, there were no known examples of unbalanced Wilf-equivalences between finite sets of patterns, though now many have been discovered, as described in \cite{Bloom} and \cite{Burstein}.

Previous work has demonstrated the value in using statistics to describe the structure of permutations. In \cite{Pudwell}, statistics concerned with relative order of consecutive entries including double ascents, double descents, peaks, and valleys were used to compare structures of Wilf-equivalent sets of permutations such as $\Av(\sigma)$ where $\sigma$ is a pattern of length 3. In \cite{Goyt}, the statistic inversion number is used to examine the structure of Fibonacci permutations. This statistic is then encoded into generating functions to find analogs of Fibonacci identities from which we get the $q$-Fibonacci numbers in \cite{Goyt} and \cite{GoytSagan}. 

In this paper, we follow \cite{Goyt} and use inversion number to describe the structure of almost-Fibonacci sets of pattern avoiding permutations. We also consider a statistic that gives the length of the ending Fibonacci subsequence. We encode these two statistics in generating functions which have the form
$$G_n^{A_1}(v,q) = \sum_{\pi \in \Av_n(A_1)}v^{\Fib(\pi)}q^{\inv(\pi)}.$$
which we will write in terms of those in \cite{Goyt}. Finally, we use these generating functions to derive analogs of Fibonacci identities similar to those in \cite{Goyt}.

The computational tools involving generating trees to demonstrate this enumeration come from \cite{Vatter} and are well-known. The novel parts of our paper are the structural descriptions of these sets of permutations and the proofs of their enumeration using them.

\section{Enumeration}
Before we prove any results about our sets of pattern-avoiding permutations, we describe their structure. We employ the following definitions for simplicity.

\begin{definition}
If a permutation $\pi$ has length $n$ and a permutation $\sigma$ has length $k$, then we define $\pi \oplus \sigma$ as $\pi$ followed by $\sigma$, where all of the entries of $\sigma$ have been increased by $n$. Similarly, we define $\pi \ominus \sigma$ as $\pi$ followed by $\sigma$, where all of the entries of $\pi$ have been increased by $k$.
\end{definition}

For example, if $\pi = 132$ and $\sigma = 312$, then $\pi \oplus \sigma = 132645$ and $\pi \ominus \sigma = 465312$. Using this definition, we describe permutations in the set $\Av(A_1).$

\begin{theorem}
The permutations in the set $\Av(A_1)$ are exactly the permutations of the form $\pi \oplus \sigma \oplus \tau$, where $\pi$ is an increasing permutation that can be empty, $\sigma$ is either empty or is $321$, and $\tau$ is Fibonacci.
\end{theorem}

In Figure \ref{fig:AvA1} we have a set of points for which $\Av(A_1)$ is the grid class assuming  even if the line segment has finite length, we can choose arbitrarily many points and there are infinitely many pairs of points in the upper right square. Note that if only one of the points in the middle block is chosen, the middle block becomes part of the first block, and if two of them are chosen, the middle block becomes part of the last block. No matter which collection of points in the image is chosen, the corresponding permutation will still be in $\Av(A_1)$.

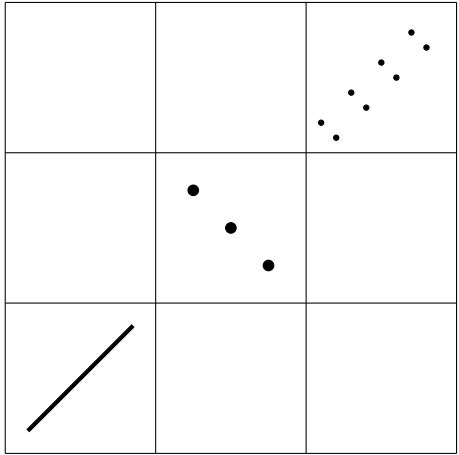
\begin{figure}[ht]
\centering
\begin{tikzpicture}
\draw (0,0) -- (0,6);
\draw (0,0) -- (6,0);
\draw (0,6) -- (6,6);
\draw (6,0) -- (6,6);
\draw (2,0) -- (2,6);
\draw (4,0) -- (4,6);
\draw (0,2) -- (6,2);
\draw (0,4) -- (6,4);
\draw[black, ultra thick] (0.3,0.3) -- (1.7,1.7);
\filldraw[black] (2.5,3.5) circle (2pt);
\filldraw[black] (3,3) circle (2pt);
\filldraw[black] (3.5,2.5) circle (2pt);
\filldraw[black] (4.4,4.2) circle (1pt);
\filldraw[black] (4.2,4.4) circle (1pt);
\filldraw[black] (4.8,4.6) circle (1pt);
\filldraw[black] (4.6,4.8) circle (1pt);
\filldraw[black] (5,5.2) circle (1pt);
\filldraw[black] (5.2,5) circle (1pt);
\filldraw[black] (5.4,5.6) circle (1pt);
\filldraw[black] (5.6,5.4) circle (1pt);
\end{tikzpicture}
\caption{ $\Av(A_1)$ is the grid class for this set}
\label{fig:AvA1}
\end{figure}

\begin{proof}
We consider two types of permutations in $\Av(A_1)$: permutations with a 321 pattern and permutations without a 321 pattern.

\noindent Case 1: $\pi$ contains 321.

In this case, we place the 321 in an arbitrary position and consider what can happen around the 321. We first claim the entries in the 321 must be consecutive in value. If they were not consecutive then there would exist some entry of $\pi$ with a value between that of two decreasing entries on either the left or the right, giving either a 231 or a 312. Additionally, the entries in the 321 must be consecutive in placement. If an entry less than all of the entries in the 321 were in the middle of the 321 then we would get a 312, and if an entry greater than all of the entries in the 321 were in the middle of the 321, then we would get a 231. Thus, the 321 subsequence is consecutive in placement and value. \par

The subsequence of $\pi$ to the left of the 321 must be less than the entries in the 321 in order to avoid 4321. Additionally, it must be increasing in order for $\pi$ to avoid 21543. \par 
The the entries in the subsequence to the right of the 321 (which we called $\tau$) must be greater than the entries in the 321 in order for $\pi$ to avoid 4321.  Additionally, $\tau$ must avoid 321 in order for $\pi$ to avoid 21543. Thus, $\tau$ must avoid 321, 312 and 231, so it is Fibonacci. This shows that if a permutation has a 321 than it must have the stated structure in order to be in the set.

\noindent Case 2: $\pi$ avoids 321.

In this case, $\pi$ avoids 321 in addition to 231 and 312. Note $\pi$ also avoids 21543 and 4321 since they both contain 321. Thus, any permutation in this set is a Fibonacci permutation, which means it has the form we want with $\pi$ and $\sigma$ empty.

Now, we must show that permutations with the structure we described avoid the patterns in $A_1$. Again, $\tau$ avoids 312, 321 and 231, so it avoids all of the necessary patterns. Additionally, the entries in $\tau$ is greater than all of the other entries in the permutation so there is no way that part of it can be in one of the forbidden patterns. Furthermore, $\pi$ is an increasing permutation and $\sigma$ is 321. Both of these subsequences avoid all of the forbidden patterns separately, and any subsequence of $\pi \oplus \sigma$ will also avoid all of the forbidden patterns. Thus, the structure avoids all of the forbidden patterns.
\end{proof}

As we show next, permutations in $\Av(A_2)$ have a very similar structure with 321 replaced with 312.
\begin{theorem}
\label{A2Structure}
The permutations in the set $\Av(A_2)$ are exactly the permutations of the form $\pi \oplus \sigma \oplus \tau$, where $\pi$ is an increasing permutation that can be empty, $\sigma$ is either empty or is $312$, and $\tau$ is Fibonacci.
\end{theorem}

Theorem \ref{A2Structure} tells us $\Av(A_2)$ is the grid class associated with the graph in Figure \ref{fig:AvA2} again assuming even if the line segment has finite length, we can choose arbitrarily many points and the pairs in the upper right box continue infinitely.
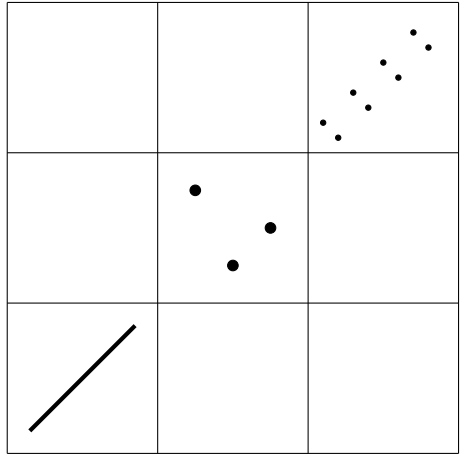
\begin{figure}[ht]
\centering
\begin{tikzpicture}
\draw (0,0) -- (0,6);
\draw (0,0) -- (6,0);
\draw (0,6) -- (6,6);
\draw (6,0) -- (6,6);
\draw (2,0) -- (2,6);
\draw (4,0) -- (4,6);
\draw (0,2) -- (6,2);
\draw (0,4) -- (6,4);
\draw[black, ultra thick] (0.3,0.3) -- (1.7,1.7);
\filldraw[black] (2.5,3.5) circle (2pt);
\filldraw[black] (3,2.5) circle (2pt);
\filldraw[black] (3.5,3) circle (2pt);
\filldraw[black] (4.4,4.2) circle (1pt);
\filldraw[black] (4.2,4.4) circle (1pt);
\filldraw[black] (4.8,4.6) circle (1pt);
\filldraw[black] (4.6,4.8) circle (1pt);
\filldraw[black] (5,5.2) circle (1pt);
\filldraw[black] (5.2,5) circle (1pt);
\filldraw[black] (5.4,5.6) circle (1pt);
\filldraw[black] (5.6,5.4) circle (1pt);
\end{tikzpicture}
\caption{ $\Av(A_2)$ is the grid class for this set.}
\label{fig:AvA2}
\end{figure}

\begin{proof}
We first show that if $\pi \in \Av(A_2)$, then $\pi$ has the given form. To do this, we consider all possible permutations in $\Av(A_2)$ in two cases: permutations with a 312 pattern and permutations without a 312 pattern. \par

\noindent Case 1: $\pi$ contains 312.

In this case, we place a 312 in an arbitrary position in $\pi$ and consider what can happen around it. We first show that the entries in the 312 subsequence must be consecutive in value. If they were not consecutive then there would exist either a 321, 231 or a 4123. Additionally, the 312 must be consecutive in placement. If an entry less than all of the entries in the 312 were in the middle of the 312 than we would get a 321 or a 4123, and if the entry greater than all of the entries in the 321 were in the middle of the 312 then we would get a 231. Thus, the 312 subsequence is consecutive in placement and value. \par

The subsequence of $\pi$ to the left of the 312 (which we called $\pi$) must be less than the entries in the 312 in order for $\pi$ to avoid 321. Additionally, $\pi$ must be monotone increasing in order for $\pi$ to avoid 21534. The entries in the subsequence of $\pi$ to the right of the 312 (which we called $\tau$) must be greater than the entries in the 312 in order to avoid 321. Additionally, $\tau$ must avoid 312 in order for $\pi$ to avoid 21534. Thus, $\tau$ avoids 321, 312 and 231, so it is Fibonacci. \par

\noindent Case 2: $\pi$ avoids 312.

In this case, $\pi$ avoids 312 in addition to 231 and 321. Note that the other two patterns will automatically be avoided since they both have a 312 in them. Thus, $\pi$ is a Fibonacci permutation, which means it has the form we want with $\pi$ and $\sigma$ empty.

Next, we must show that permutations with the structure we described avoid the necessary patterns. Again, $\tau$ avoids 312, 321 and 231, so it avoids all of the necessary patterns. Additionally, the entries of $\tau$ are greater than all of the other entries in the permutations so there is no way that part of it can be in one of the forbidden patterns. Furthermore, $\pi$ is an increasing permutation and $\sigma$ is a 312. Both of these subsequences avoid all of the forbidden patterns separately, and any subsequence of $\pi \oplus \sigma$ will also avoid all of the forbidden patterns. Thus, the structure avoids all of the forbidden patterns.
\end{proof}

Now that we have described the structure of permutations in $\Av(A_1)$ and $\Av(A_2)$, we can enumerate them.

\begin{theorem}\label{2.4} 
For all $n \geq 1$,
$$|\Av(A_1)| = |\Av(A_2)| = F_{n+1} - 1.$$
\end{theorem}

We will give two proofs of this theorem. The first reduces permutations to Fibonacci permutations and relies on a well-known identity of the Fibonacci numbers given by
\begin{equation}
\label{Fibonacci Identity}
\sum_{k=1}^n F_k = F_{n+2}-1.
\end{equation}
The second uses bijections to domino and monomino tilings of a $1 \times (n+1)$ board.

\begin{proof}
We divide permutations of length $n$ in $\Av(A_1)$ into two types. The first type are the permutations that do not have a 321. These permutations are the Fibonacci permutations of length $n$ so there are $F_n$ of them. The second type are the permutations that have a 321. Since $\pi$ and $\sigma$ are fixed for a given length of $\pi$, we sum the possible subsequences $\tau$ over every possible length of $\tau$, which can be anywhere from 0 to $n-3$. Summing over both cases and using \eqref{Fibonacci Identity} gives
\begin{align*}
    |\Av_n(A_1)| &= F_n + \sum_{j=0}^{n-3}F_j\\
    &= F_n + F_{n-1} - 1\\
    &= F_{n+1} - 1.
\end{align*}
The proof for $\Av(A_2)$ is identical except it is based on the presence and position of the 312 instead of the 321.
\end{proof}

Next, recall that the tilings of a $1\times n$ board with dominoes and monominoes are counted by the Fibonacci numbers (see \cite{Benjamin}). Thus, we can also prove the identity with a bijection to tilings of a $1\times (n+1)$ board with dominoes and monominoes that excludes a single tiling. Before we do this, we must briefly explain the bijection between Fibonacci permutations of length $n$ and tilings of a $1 \times n$ board with dominoes and monominoes.

Using Goyt and Mathisen's \cite{Goyt} description of the structure of a Fibonacci permutation $\pi = \pi_1\pi_2\cdots\pi_k$ where $\pi_i < \pi_j$ whenever $i < j$ and each $\pi_i$ is a decreasing sequence of at most two entries (see Figure \ref{fig:Fibonacci}), we map every $\pi_i$ of length one to a monomino and every $\pi_i$ of length two to a domino. We will call this mapping $\Phi$. Some examples of this mapping are $\Phi(1324576) = mdmmd$ and $\Phi^{-1}(dmdmm) = 2135467$.

For our bijection, we will map every permutation in $\Av(A_1)$ and $\Av(A_2)$ to a unique tiling of a $1\times(n+1)$ board with dominoes and monominoes excluding the tiling that starts with a domino and is followed only by monominoes.

\begin{proof}[Alternate proof of \textbf{Theorem \ref{2.4}}]
We define a function $\phi$ from permutations in $\Av(A_1)$ to tilings described above. As in the proof Theorem 2.2, we define $\phi$ in two cases based on the presence or absence of a 321. If there is not a 321 in a permutation, then $\phi$ will place a monomino at the beginning of the tiling and tile the remaining $1\times n$ board using $\Phi$. Note that since the permutation avoids 321, it is entirely Fibonacci, so this makes sense. If there is a 321, then $\phi$ will place a domino at the beginning, remove the 1 in the 321, and tile the remaining $1\times (n-1)$ board using $\Phi$ and the remaining $n-1$ entries in the permutation. We can describe $\phi$ recursively as
$$\phi(\pi \oplus \sigma \oplus \tau) = \begin{cases}
m \oplus \Phi(\tau) & \sigma = \emptyset, \\
d \oplus \Phi(\pi \oplus 21 \oplus \tau) & \sigma \neq \emptyset,\\
\end{cases}$$
where $\Phi$ is the bijection to Fibonacci tilings, $m$ is a monomino, $d$ is a domino, and $\oplus$ is concatenation of tilings. Note that this function will not produce the tiling given by a domino followed by only monominoes because the domino at the beginning means there is a 321 and thus, when we do not consider the 1, we still have a descent that will be mapped to a domino, giving us a domino in the middle in addition to the initial domino. Some examples of $\phi$ are $\phi(214356) = mddmm$ and $\phi(143265) = dmdd$. \par

The inverse of $\phi$ can be described in a similar piecewise manner: for a tiling $x\oplus y$, where $x$ is a single tile, we have
$$\phi^{-1}(x\oplus y) = \begin{cases}
\Phi^{-1}(y) & \text{if }x = m \\
\psi(y) & \text{if }x = d.\\
\end{cases}$$
Here, $\psi$ is the function that is identical to $\Phi^{-1}$ except it sends the first domino to a consecutive 321 instead of a consecutive 21. Since $\phi$ and its inverse are both functions, $\phi$ is a bijection. Thus, the number of permutations of length $n$ in the set $\Av(A_1)$ is $F_{n+1}-1$.
\end{proof}

The tiling bijection for $\Av(A_2)$ is nearly identical, using a piecewise function $\phi$ based on the presence of a 312 instead of a 321, and a $\psi$ that maps the first domino to a 312 instead of a 321.

Now, we discuss the other two sets of pattern avoiding permutations.

\begin{theorem}
The permutations in the set $\Av(B_1)$ are exactly the permutations of the form $(\pi \ominus 1) \oplus \sigma$, where $\pi$ is a decreasing permutation and $\sigma$ is Fibonacci.
\end{theorem}

Before we prove this, we note that the structure of the permutations in this set could be written more simply as $\alpha \oplus \sigma$, where $\alpha$ is a decreasing permutation and $\sigma$ is Fibonacci. This simplification is visible in Figure \ref{fig:AvB1}. Since $\alpha$ is decreasing and every entry in $\alpha$ is less than every entry in $\sigma$ it must be the case that the last entry in $\alpha$ is 1, showing that this is identical to $\pi \ominus 1$. We have chosen to write it in a more complicated way to emphasize the similarity to $\Av(B_2)$. 

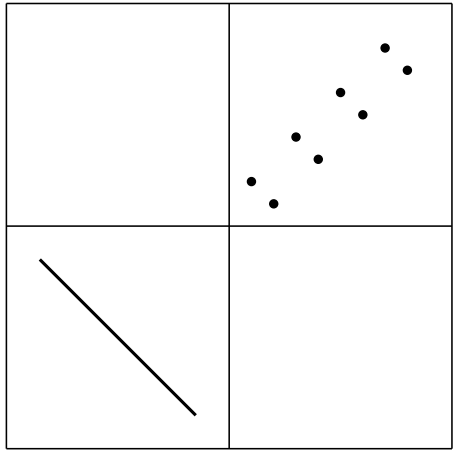
\begin{figure}[ht]
\centering
\resizebox{0.36\textwidth}{!}{\begin{tikzpicture}
\draw (0,0) -- (0,4);
\draw (0,0) -- (4,0);
\draw (0,4) -- (4,4);
\draw (4,0) -- (4,4);
\draw (2,0) -- (2,4);
\draw (0,2) -- (4,2);
\draw[black, thick] (0.3,1.7) -- (1.7,0.3);
\filldraw[black] (2.4,2.2) circle (1pt);
\filldraw[black] (2.2,2.4) circle (1pt);
\filldraw[black] (2.8,2.6) circle (1pt);
\filldraw[black] (2.6,2.8) circle (1pt);
\filldraw[black] (3,3.2) circle (1pt);
\filldraw[black] (3.2,3) circle (1pt);
\filldraw[black] (3.4,3.6) circle (1pt);
\filldraw[black] (3.6,3.4) circle (1pt);
\end{tikzpicture}}
\caption{Grid Class $\Av(B_1)$}
\label{fig:AvB1}
\end{figure}
\begin{proof}

To show all permutations in $\Av(B_1)$ have this structure, suppose $\pi\in\Av(B_1)$. Because $\pi$ avoids 312, every entry to the left of 1 is less than every entry to the right of 1. Because $\pi$ avoids 231, the entries to the left of 1 are in decreasing order. Because $\pi$ also avoids 1432, the entries to the right of 1 avoid 321, in addition to 231 and 312. Hence, the entries to the right of 1 form a Fibonacci permutation. As a result, $\pi$ has the claimed form.\par

Now, we must show that permutations that have the structure we described avoid the necessary patterns. Note that $\sigma$ is Fibonacci so it avoids 231, 312, and 321 and thus it must avoid the necessary patterns. Additionally, in our structure, the entries $\sigma$ must be greater than the rest of the entries in the permutation. Thus, there is no way to have a 231 or a 312 that is partially contained in $\sigma$, and since $\sigma$ avoids 321 there is no way to get a 1432 using this part either. Thus, one of the forbidden patterns would have to occur in $\pi \ominus 1$. However, this subsequence is strictly decreasing so it must avoid 231, 312 and 1432. This covers all cases and shows that permutations with the structure we described avoid the necessary patterns.
\end{proof}

Permutations in the $\Av(B_2)$ have a similar structure, except they start with an increasing subsequence instead of a decreasing subsequence.

\begin{theorem}
The permutations in the set $\Av(B_2)$ are exactly the permutations of the form $(\pi \ominus 1) \oplus \sigma$, where $\pi$ is an increasing permutation and $\sigma$ is Fibonacci.
\end{theorem}

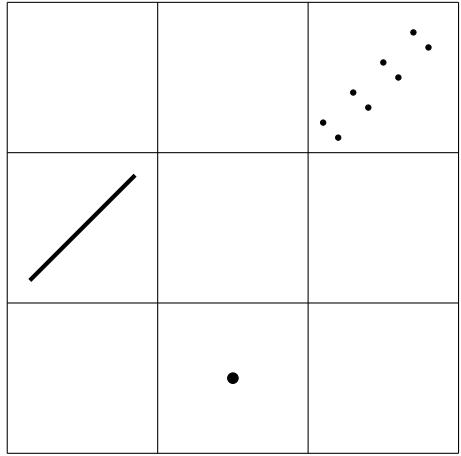
\begin{figure}
    \centering
    \begin{tikzpicture}
\draw (0,0) -- (0,6);
\draw (0,0) -- (6,0);
\draw (0,6) -- (6,6);
\draw (6,0) -- (6,6);
\draw (2,0) -- (2,6);
\draw (4,0) -- (4,6);
\draw (0,2) -- (6,2);
\draw (0,4) -- (6,4);
\draw[black, ultra thick] (0.3,2.3) -- (1.7,3.7);
\filldraw[black] (3,1) circle (2pt);
\filldraw[black] (4.4,4.2) circle (1pt);
\filldraw[black] (4.2,4.4) circle (1pt);
\filldraw[black] (4.8,4.6) circle (1pt);
\filldraw[black] (4.6,4.8) circle (1pt);
\filldraw[black] (5,5.2) circle (1pt);
\filldraw[black] (5.2,5) circle (1pt);
\filldraw[black] (5.4,5.6) circle (1pt);
\filldraw[black] (5.6,5.4) circle (1pt);
\end{tikzpicture}
    \caption{Grid Class $\Av(B_2)$}
    \label{fig:AvB2}
\end{figure}

In Figure $\ref{fig:AvB2}$ we have a set of points for which $\Av(B_2)$ is the grid class. Note that the only difference between this grid class and the grid class $\Av(B_1)$ is that the leftmost block has an increasing subsequence instead of a decreasing subsequence.

\begin{proof}
Again, we place 1 in an arbitrary position and consider the structure of the permutation around it. Since the permutation avoids 312, it must be the case that everything to the left 1 is less than everything to the right of 1. Since the permutation avoids 321, it must be the case that everything to the left of the 1 is increasing. Since the permutation avoids 1342, it must be the case that everything to the right of 1 avoids 231. Since the subsequence to the right of 1 must also avoid 321 and 312, it must be Fibonacci. Thus, the only possible structure for permutations in this set is the one that we have described. \par

 Now, we must show that permutations that have the structure we described avoid the necessary patterns. Again, $\sigma$ avoids 312, 321 and 231, so it avoids all of the necessary patterns. Additionally, since the entries in $\sigma$ are grater than all other entries in the permutation, there is no way that part of it can be in one of the forbidden patterns. Finally, $\pi \ominus 1$ must avoid 312, 321 and 1342 since $\pi$ is monotone increasing. This covers all cases and shows that permutations with the structure we described avoid the necessary patterns.
\end{proof}

Again, we can use the structure of permutations in these sets to give proofs about the enumeration of $\Av(B_1)$ and $\Av(B_2)$.

\begin{theorem}\label{2.7}
For all $n\geq 1$,
$$|\Av(B_1)| = |\Av(B_2)| = F_{n+1}-1.$$
\end{theorem}

Again, we give a proof reducing permutations in this set to Fibonacci permutations and a proof using tilings.

\begin{proof}
For $\Av(B_1)$, we divide all permutations of length $n$ in this set into two cases. In the first case, $\pi$ has length less than or equal to 1. In this case, the permutations are Fibonacci and thus there are $F_n$ of them. In the second case, $\pi$ has length at least 2. In this case, we sum over all possible lengths $k$ of $\pi$, from 2 to $n$, noting that $\pi$ is fixed and $\sigma$ is Fibonacci of length $n-k-1$. Thus, we get
$$\sum_{k=2}^n F_{n-k-1}$$
permutations. Again we can rearrange this and use $\eqref{Fibonacci Identity}$ to find that this sum is equivalent to $F_{n-1}-1$.
Summing both cases, we find again that there are
$$F_n + F_{n-1}-1 = F_{n+1}-1$$
permutations of length $n$ in the set.

The proof is nearly identical for $\Av(B_2)$.
\end{proof}

We can also give a bijection to domino and monomino tilings of a $1\times(n+1)$ board for each of the two sets, this time excluding the all monomino tiling.

\begin{proof}[Alternative proof of \textbf{Theorem \ref{2.7}}]
For $\Av(B_1)$, we construct the following function,
$$\rho: \Av(B_1) \rightarrow \{\text{Domino and monomino tilings of a } 1\times (n+1) \text{ board with at least 1 domino}\}$$
such that $\rho$ maps the entry 1 to a domino, every decreasing subsequence of length two to the right of the entry 1 to a domino, and all other entries to monominoes. Since we are mapping one entry to a tile of length two, this will increase the size of the board from $n$ to $n+1$. Additionally, it excludes the all-monomino tiling, since the entry 1 must be in the permutation, and it will always be mapped to a domino, so there must be at least one domino in all tilings in the range of our map. Some examples of $\rho$ are $\rho(3215467) = mmddmm$ and $\rho(1235476) = dmmdd$. 

To show $\rho$ is a bijection, we describe $\rho^{-1}$. We note that $\rho^{-1}$ will construct $\pi$  by mapping the first domino to 1 and the sequence of $k$ monominoes to the left of the first domino to the deceasing subsequence of length $k$ from $k+1$ to 2. Then, $\rho^{-1}$ will construct $\sigma$ using the $\Phi^{-1}$. Since this function is invertible, it is bijective. This proves that there are as many permutations in $\Av(B_1)$ as there are tilings with dominoes and monominoes of a $1\times (n+1)$ board with at least one domino. Since we know the latter is counted by $F_{n+1}-1$, the former must be as well.

The proof for $\Av(B_2)$ is nearly identical. In this case, it will be an increasing subsequence to the left of the entry 1 instead of a decreasing subsequence. But in either case every entry in that initial sequence is mapped to a monomino so there is no difference in the construction $\rho$. The inverse of $\rho$ will send all $k$ monominoes before the first domino to the consecutive increasing subsequence starting at 2 and ending at $k+1$. 
\end{proof}

\section{Enumeration By Statistics}
Next, we will count permutations in our four sets according to a couple of statistics. The first statistic we consider is inversion number, written as $\inv(\pi)$. The inversion number of a permutation $\pi$ the number of 21 patterns in $\pi$. For example, if $\pi = 15324$, then $\inv(\pi) = 4$.

\begin{proposition}
\label{FibInv}
For $n,k \in \mathbb N$ such that $n \geq k$ the number of Fibonacci permutations with $k$ inversions is given by $\binom{n-k}k$.
\end{proposition}

\begin{proof}
Use the mapping from $\Av_n(231,312,321)$ to domino and monomino tilings of a $1\times n$ board. Note that the number of inversions is given by the number of dominoes. 
\end{proof}

For our permutations, distribution of inversion number is slightly more complicated.

\begin{theorem}
For $k \geq 3$, the number of permutations of length $n$ with $k$ inversions in $\Av(A_1)$ is given by
$$\binom{n-k}{k}+\sum_{\ell=k-3}^{n-3}\binom{\ell-(k-3)}{k-3} = \binom{n-k}{k}+\binom{n-k+1}{k-2}$$
\end{theorem}

\begin{proof}
Again, we divide into cases based on the presence of a 321. The permutations that do not have a 321 are the Fibonacci permutations, so by Proposition $\ref{FibInv}$ there are $\binom{n-k}{k}$ permutations of length $n$ of this form with $k$ inversions. When there is a 321, then there are already 3 inversions contained within that subsequence. Thus, we are looking for $k-3$ inversions in the rest of the permutation. There can be no inversions in the increasing subsequence before the 321, nor can there be any inversions between the 321 and any entry before it or after it. Thus, the only place to look for additional inversions is in the part after the 321, which is Fibonacci. Note that this part can vary in length from $k-3$ to $n-3$, because  there are $k-3$ inversions in the subsequence and the upper bound is bounded by the fact that there needs to be a 321 pattern in the sequence. Thus, we need to sum the ways to get $k-3$ inversions over possible lengths $\ell$ of the Fibonacci part, giving us $$\binom{n-k}{k}+\sum_{\ell=k-3}^{n-3}\binom{\ell-(k-3)}{k-3}.$$

Using to the hockey-stick identity, 
\begin{equation}
\label{Hockey-Stick}
    \sum_{i = r}^n\binom{i}{r} = \binom{n+1}{r+1}, \text{ for } n,r\in\mathbb{N}, n>r,
\end{equation} we simplify the summation to a single binomial coefficient,
$$\sum_{\ell=k-3}^{n-3}\binom{\ell-(k-3)}{k-3} = \binom{n-k+1}{k-2}.$$
\end{proof}

The distribution of inversion number is very similar for $\Av(A_2)$.
\begin{theorem}
For $k \geq 2$, the number of permutations of length $n$ with $k$ inversions in the set $\Av(A_2)$ is given by
$$\binom{n-k}{k}+\sum_{\ell=k-2}^{n-3}\binom{\ell-(k-2)}{k-2} = \binom{n-k}{k}+\binom{n-k}{k-1} = \binom{n-k+1}{k}.$$
\end{theorem}
\begin{proof}
The proof of this is identical to the proof of the last identity except we note that the 312 pattern only has two inversions, so when it is included in the permutation, the remaining part must have $k-2$ inversions, instead of $k-3$ inversions as before. Similarly, using \eqref{Hockey-Stick} allows us to simplify the summation. This time, Pascal's rule allows us to simplify the expression even further.
\end{proof}

Note that for $0\leq k\leq 2$ for $\Av(A_1)$ and $0 \leq k \leq 1$ for $\Av(A_2)$ there will be $\binom{n-k} k$ permutations of length $n$ with $k$ inversions since there cannot be a 321 or a 312 respectively so the permutations with $k$ inversions must be Fibonacci.

Next, we consider the distribution of inversion number for $\Av(B_2)$.

\begin{theorem}
There are
$$\sum_{\ell=1}^n \binom{n-k-1}{k-\ell+1}$$
permutations of length $n$ in $\Av(B_2)$ with $k$ inversions.
\end{theorem}

\begin{proof}
We divide this proof into cases based on the length $\ell$ of $\pi \ominus 1$, and then sum over all of them. Note that in this case, it is simpler to sum over the length of the non-Fibonacci part instead of the length of the Fibonacci part. For a given length $\ell$, there are $\ell-1$ inversions in $\pi \ominus 1$ since $\pi$ is increasing and every entry in $\pi$ will have an inversion with the 1. Thus, there must be $k-(\ell-1)$ inversions in the Fibonacci part. Additionally, we know that the Fibonacci part has length $n-l$. Thus, the number of permutations that satisfy this is given by $$\binom{(n-\ell)-(k-(\ell-1))}{k-(\ell-1)} = \binom{n-k-1}{k-\ell+1}.$$
Summing over all possible lengths $\ell$ gives us the desired result.
\end{proof}

Note that in this proof we could split $\Av(B_2)$ into permutations that are Fibonacci and permutations that are not. Permutations that are Fibonacci have either $\ell=1$ or $\ell=2$. Thus, we use Pascal's rule to sum the first two terms to get
$$\binom{n-k-1}{k-1} + \binom{n-k-1}{k} = \binom{n-k}{k}.$$
As expected, this is the number of Fibonacci permutations of length $n$ with $k$ inversions.

The distribution of inversion number for permutations of length $n$ in $\Av(B_1)$ follows the same logic.

\begin{theorem}
There are
$$\sum_{\ell=1}^n \binom{n-\ell-(k-\binom{\ell}{2})}{k-\binom{\ell}{2}}$$
permutations of length $n$ in $\Av(B_1)$ with inversion number $k$.
\end{theorem}

\begin{proof}
The proof of this is nearly identical to the previous proof. The only difference is there are $\binom{\ell}{2}$ inversions in $\pi \ominus 1$ of length $\ell$, since $\pi$ is decreasing so every pair of entries in this subsequence results in an inversion.
\end{proof}
We also define the following statistic to keep track of the length of the ending Fibonacci subsequence common among permutations in all four of our sets.

\begin{definition}
The statistic $\Fib(\pi)$ gives the length of the ending Fibonacci subsequence containing $n$ in a permutation of length $n$.
\end{definition}

For example, $\Fib(321) = 0$, $\Fib(132) = 2$, $\Fib(231) = 1$, $\Fib(2341657) = 3$ and $\Fib(1253467) = 2$. In this last case, it is important to note that in order to be Fibonacci as we define it, the subsequence must contain every consecutive integer from its lowest integer to $n$. This means the 3467 subsequence is not Fibonacci so 67 is the longest Fibonacci subsequence ending at the last entry of the sequence. We can now consider the distribution of this statistic over permutations in each set.

\begin{theorem}
There are $F_k$ permutations $\pi$ of length $n$ in $\Av(A_1)$, $\Av(A_2), \Av(B_1)$, and $\Av(B_2)$ with $\Fib(\pi) = k$. 
\end{theorem}

\begin{proof}
For each set, given the length $k$ of the ending Fibonacci subsequence concluding $n$, everything before it is fixed. Thus, the only variation in the permutations occurs in the Fibonacci part, resulting in the expected $F_k$ permutations.
\end{proof}

We can combine these enumerations to get a general formula for the number of permutations $\pi$ of length $n$ in our set such that $\inv(\pi) = j$ and $\Fib(\pi) = k$.

\begin{theorem}
For $k < n$ and $j \in \mathbb N$, there are
$$\binom{k-j+3}{j-3}, \binom{k-j+2}{j-2}, \binom{k-j+\binom{n-k}{2}}{j-\binom{n-k}{2}}, \text{ and } \binom{2k+j+1-n}{j+k+1-n}$$
permutations $\pi$ of length $n$ such that $\inv(\pi)=j$ and $\Fib(\pi) = k$ in $\Av(A_1)$, $\Av(A_2)$, $\Av(B_1)$, and $\Av(B_2)$, respectively, as long as the binomial coefficients are defined. If they are not defined, then there are no permutations with those values of $k$ and $j$. If $k = n$, then there will be $\binom{k-j}{j}$ permutations $\pi$ such that $\inv(\pi) = j$ and $\Fib(\pi) = k$.

\end{theorem}

\begin{proof}
The results follow from determining how many inversions must be in the Fibonacci subsequence of each permutation. If $\pi \in Av_n(A_1)$ has $\inv(\pi) = j$ and $\Fib(\pi) = k < n$, then $\pi$ is an increasing sequence of length $n-k-3$, followed by a 321 pattern, followed by a Fibonacci permutation of length $k$. The first two parts contain a total of 3 inversions, so the Fibonacci part must contain $j-3$ inversions. By Proposition \ref{FibInv} there are $\binom{k-j+3}{j-3}$ such permutations. The remaining cases are similar, but the inversion number of the pre-Fibonacci parts are given by 2, $\binom{n-k}{2}$ and $n-k-1$ for $\Av(A_2)$, $\Av(B_1)$, and $\Av(B_2)$, respectively. In any of the cases, if $k = n$, then $\pi$ is Fibonacci so there will be $\binom{k-j}{j}$ permutations with $\inv(\pi) = j$ and $\Fib(\pi) = k$ by Proposition \ref{FibInv}.
\end{proof}

\section{Generating Functions}
In this section, we encode the statistics $\inv$ and $\Fib$ in generating functions similar to those in $\cite{Goyt}$. We define the generating function $G_n$ for the set $\Av(A_1)$ as
$$G_n^{A_1}(v,q) = \sum_{\pi \in \Av_n(A_1)}v^{\Fib(\pi)}q^{\inv(\pi)}.$$
We define generating functions for the sets $\Av(A_2)$, $\Av(B_1)$, and $\Av(B_2)$ in the same way such that
$$G_n^{A_2}(v,q) = \sum_{\pi \in \Av_n(A_2)}v^{\Fib(\pi)}q^{\inv(\pi)},$$
$$G_n^{B_1}(v,q) = \sum_{\pi \in \Av_n(B_1)}v^{\Fib(\pi)}q^{\inv(\pi)},$$
and
$$G_n^{B_2}(v,q) = \sum_{\pi \in \Av_n(B_2)}v^{\Fib(\pi)}q^{\inv(\pi)}.$$

Now that we have defined the generating functions, we can describe their relationship with the Fibonacci generating function defined as
$$F_n(q) = \sum_{\pi \in \Av_n(231,312,321)}q^{\inv(\pi)}.$$

\begin{theorem}
For all $n \geq 3$,
$$G_n^{A_1}(v,q) = F_n(q)v^n + \sum_{j=0}^{n-3} q^3v^j F_{j}(q),$$
$$G_n^{A_2}(v,q) = F_n(q)v^n + \sum_{j=0}^{n-3} q^2v^j F_j(q),$$
$$G_n^{B_1}(v,q) = F_n(q)v^n + \sum_{j=0}^{n-3} q^{\binom{j}{2}}v^j F_j(q),$$
and
$$G_n^{B_2}(v,q) = F_n(q)v^n + \sum_{j=0}^{n-3} q^{j-1}v^j F_j(q).$$
\end{theorem}

\begin{proof}
We start with $G_n^{A_1}(v,q)$, intending to sum over all possible values of Fib. When $\Fib(\pi) = n$, the permutation is Fibonacci. This accounts for the $F_n(q)v^n$ term. When $\Fib(\pi) = j < n$, the permutation is not Fibonacci. In this case, that means it has a 321 after an increasing subsequence $\pi$. The 321 will have three inversions and the Fibonacci subsequence of length $j$ will contribute $F_{j}(q)$ inversions. Note that $j$ must be less than $n-3$ since the 321 is not part of $\sigma_3$ and it has length three. Summing over all possible values of $j$ and adding the Fibonacci permutations will give a generating function equivalent to $G_n^{A_1}$.

The proofs of the other three equations are almost identical except for the differences in counting inversions for the pre-Fibonacci part. For $G_n^{A_2}$ there are exactly two inversions in the pre-Fibonacci part coming from the 312. For $G_n^{B_1}$, there are $\binom{j}{2}$ inversions coming from the decreasing subsequence of length $j$. For $G_n^{B_2}$, there are $j-1$ inversions coming from the increasing sequence of length $j-1$ followed by the entry 1.
\end{proof}

We can also describe recurrence relations for these generating functions. Before we do this, we describe the recurrence relation for the terms in our sequence.

\begin{theorem}
If $a_n = F_{n+1}-1$, then for $n\geq 2$, we have
$$a_n = a_{n-1}+a_{n-2}+1.$$
\end{theorem}
\begin{proof}
To show this, we work in terms of Fibonacci numbers and use their recurrence relation:
\begin{align*}
a_n &= F_{n+1}-1 \\
&= F_n + F_{n-1} -1 \\
&= a_{n-1} + F_{n-1} \\
&= a_{n-1}+a_{n-2}+1.
\end{align*}
\end{proof}

We can see combinatorially why this recurrence relation makes sense in a similar way to the combinatorial proof for the Fibonacci recurrence relation. We will use $\Av(A_1)$ as an example, but the others work similarly. Permutations in $\Av(A_1)$ end in a decreasing subsequence of length one, two, or three. If the decreasing subsequence has length one or two, it is in the Fibonacci part $\sigma$. Thus, like with Fibonacci permutations, we can remove it. This will leave a permutation in $\Av(A_1)$ of length $n-1$ or $n-2$. If the decreasing subsequence has length three, then it must be the 321. There is only one possible permutation in this case, given by $\tau \oplus 321$ where $\tau$ is strictly increasing. Thus, there are $a_{n-1}$ permutations that end in a decreasing subsequence of length one, $a_{n-2}$ permutations that end in a decreasing subsequence of length two, and one permutation that ends in a decreasing subsequence of length three, so $a_n = a_{n-1} + a_{n-2} + 1$.

In a similar way, we can find recurrence relations for our generating functions.

\begin{theorem}
For $n \geq 2$,
$$G_n^{A_1}(v,q) = q^3 + vG_{n-1}^{A_1}(v,q) + qv^2G_{n-2}^{A_1}(v,q),$$
$$G_n^{A_2}(v,q) = q^2 + vG_{n-1}^{A_2}(v,q) + qv^2G_{n-2}^{A_2}(v,q),$$
$$G_n^{B_1}(v,q) = q^{\binom{n}{2}} + vG_{n-1}^{B_1}(v,q) + qv^2G_{n-2}^{B_1}(v,q),$$
and
$$G_n^{B_2}(v,q) = q^{n-1} + vG_{n-1}^{B_2}(v,q) + qv^2G_{n-2}^{B_2}(v,q).$$
\end{theorem}

\begin{proof}
We start with $G_n^{A_1}(v,q)$. Let $\pi \in \Av(A_1)$ have length $n$. Given the structure of permutations in $\Av(A_1)$, $\pi$ can end in a decreasing subsequence of length one, two, or three. If it ends in a decreasing subsequence of length 1, it corresponds to a permutation of length $n-1$ with $n$ added to the end. Thus, it has the same number of inversions and the value of Fib increases by 1 compared to the corresponding permutation of length $n-1$. Similarly, if it ends in a decreasing subsequence of length two, it corresponds to a permutation of length $n-2$ with $n, n-1$ added to the end. Thus, it has one more inversion and Fib increases by two compared to the corresponding permutation of length $n-2$ in $\Av(A_1)$. If it ends in a decreasing subsequence of length 3, then it must be the permutation given by $1,2,3,\ldots,n,n-1,n-2$. Thus, $\Fib(\pi) = 0$ and $\inv(\pi) = 3$. Summing together these three possibilities results in the equality
$$G_n^{A_1}(v,q) = q^3 + vG_{n-1}^{A_1}(v,q) + qv^2G_{n-2}^{A_1}(v,q).$$

The other three proofs are similar. In each case, we get the same terms for the permutations that end in a consecutive decreasing subsequence of length one or two. The difference lies in the permutations where $\sigma$ is empty. Note that there is only one of these in each set with a given length $n$. For $\Av(A_2)$, this permutation is given by $\pi =1,2,3,\ldots,n,n-2,n-1$. In this case $\inv(\pi) = 2$ and $\Fib(\pi) = 0$. For $\Av(B_1)$, $\pi$ is the decreasing permutation. In this case $\inv(\pi) = \binom{n}{2}$ and $\Fib(\pi) = 0$. For $\Av(B_2)$, this permutation is given by $\pi =2,3,\ldots,n,1$. In this case $\inv(\pi) = n-1$ and $\Fib(\pi) = n-1$. This gives us the first terms for each of the recurrence relations.
\end{proof}

Beyond the recurrence relation, we can prove other analogs of Fibonacci identities using our generating functions. For example, we consider the identity
$$F_{n+m} = F_{n-1}F_m + F_nF_{m-1}.$$

\begin{theorem}
For $m,n \geq 2$,
\begin{align*}G_{m+n}^{A_1}(v,q) = v^nG_m^{A_1}(v,q)F_n(q) &+ v^{n+2}qG_{m-1}^{A_1}(v,q)F_{n-1}(q) + v^{n-1}q^3F_{n-1}(q) \\ &+ v^{n-2}q^3F_{n-2}(q) +  G_n^{A_1}(v,q)-v^nF_n(q),
\end{align*}

\begin{align*}
    G_{m+n}^{A_2}(v,q) = v^nG_m^{A_2}(v,q)F_n(q) &+ v^{n+2}qG_{m-1}^{A_2}(v,q)F_{n-1}(q) + v^{n-1}q^2F_{n-1}(q) \\ &+ v^{n-2}q^2F_{n-2}(q) + G_n^{A_2}(v,q)-v^nF_n(q),
\end{align*}

$$G_{m+n}^{B_1}(v,q) = v^nG_m^{B_1}(v,q)F_n(q) + v^{n+2}qG_{m-1}^{B_1}(v,q)F_{n-1}(q) + v^{\binom{m}{2}}\sum_{i=0}^nv^iF_i(q)q^{\binom{i}{2}+im},
$$

\begin{align*}
G_{m+n}^{B_2}(v,q) =  v^nG_m^{B_2}(v,q)F_n(q) &+ qv^{n+2}G_{m-1}^{B_2}(v,q)F_{n-1}(v,q) + q^mv^{n-1}F_{n-1}(q) \\ &+ q^m(G_n^{B_2}(v,q)-v^nF_n(q)).
\end{align*}
\end{theorem}

This analog is somewhat messier because there are more cases depending on the placement of the ``cut" of a permutation of length $m+n$ into a permutation of length $m$ and a permutation of length $n$. This reflects the additional structural complexity of permutations in our sets compared to Fibonacci permutations. $B_1$ is especially messy due to the contrasted decreasing initial subsequence and increasing Fibonacci subsequence.

\begin{proof}
For $A_1$, there are five cases that we will sum over, depending on where a permutation $\pi$ of length $m+n$ is divided into $\pi_m$ and $\pi_n$ such that $\pi = \pi_m\pi_n$. In the first two cases, $\pi_n$ is Fibonacci. In the first case, the last entry in $\pi_m$ is less than the first entry in $\pi_n$. Thus, we get the term $v^nG_m^{A_1}(v,q)F_n(q)$. Note that this includes the Fibonacci permutations. In the second case, the last entry in $\pi_m$ is greater than the first entry in $\pi_n$, which is to say that the division was in the middle of a descent. In this case, we pull the descent out, noting that it will add 1 inversion and 2 to Fib, to get the term $v^{n+2}qG_{m-1}^{A_1}(v,q)F_{n-1}(q).$

In the third case, $\pi_m = 1,2,3,\ldots,m-2,m+1,m$ and $\pi_n$ starts with $m-1$ and the rest of it is Fibonacci. Thus, $\Fib(\pi) = n-1$ and the inversions are counted by the three in the 321 added to those in the ending Fibonacci subsequence of length $n-1$ resulting in the term $v^{n-1}q^3F_{n-1}(q)$.

In the fourth case, $\pi_m = 1,2,3,\ldots,m-1,m+2$ and $\pi_n$ starts with $m+1,m$ and the rest of it is Fibonacci. In this case, we get $\Fib(\pi) = n-2$ and three inversions before the ending Fibonacci subsequence of length $n-2$, resulting in the term $v^{n-2}q^3F_{n-2}(q)$.

In the fifth case, $\pi_m$ is an increasing subsequence. Thus, we have get the term $G_n^{A_2}(v,q)$. However, this will double count Fibonacci permutations, so we have to subtract out the $v^nF_n(q)$. Note that since we are only considering permutations that are not Fibonacci for $\pi_n$, we do not have to worry about $\pi_m$ in our calculation of Fib.

Summing all of the possible cases together gives us the desired result, 
\begin{align*}
    G_{m+n}^{A_1}(v,q) = v^nG_m^{A_1}(v,q)F_n(q) &+ v^{n+2}qG_{m-1}^{A_1}(v,q)F_{n-1}(q) + v^{n-1}q^3F_{n-1}(q) \\ &+ v^{n-2}q^3F_{n-2}(q) +  G_n^{A_1}(v,q)-v^nF_n(q).
\end{align*}

A similar process will yield similar results for the other sets.
\end{proof}

\section{Acknowledgements}
We would like to thank Professor Jay Pantone for his computational work that hypothesized the Wilf-equivalence of $A_1$, $A_2$, $B_1$, and $B_2$ as well as our research advisor, Professor Eric Egge, for his guidance and support.

\bibliography{Bibliography}{}
\bibliographystyle{plain}

\end{document}